\documentclass[12pt,reqno]{amsart}
\usepackage{amsmath, amsthm, amsopn, amssymb, microtype}
\usepackage{mathrsfs} 
\usepackage{mathscinet}
\usepackage{bbm}
\usepackage{enumerate, tikz, etoolbox, intcalc, geometry, caption, subcaption, afterpage}
\usepackage[english]{babel}
\usepackage{todonotes}

\title{Automatic continuity of Polynomial maps and cocycles} 
\author{Tom Meyerovitch}
\address{Ben Gurion University of the Negev.
	Departement of Mathematics.
	Be'er Sheva, 8410501, Israel
}

\author{Omri Nisan Solan}
\address{The Hebrew University of Jerusalem.
Einstein Institute of Mathematics
Edmond J. Safra Campus
Givat Ram. Jerusalem, 9190401, Israel
}

\usepackage[colorlinks=true,urlcolor=blue,pdfborder={0 0 0}]{hyperref}
\hypersetup{linkcolor=[rgb]{0,0,0.6}}
\hypersetup{citecolor=[rgb]{0,0.6,0}}
\usepackage[nameinlink]{cleveref}
\crefname{theorem}{Theorem}{Theorems}
\crefname{thm}{Theorem}{Theorems}
\crefname{mainthm}{Theorem}{Theorems}
\crefname{lemma}{Lemma}{Lemmas}
\crefname{lem}{Lemma}{Lemmas}
\crefname{remark}{Remark}{Remarks}
\crefname{prop}{Proposition}{Propositions}
\crefname{defn}{Definition}{Definitions}
\crefname{corollary}{Corollary}{Corollaries}
\crefname{cor}{Corollary}{Corollaries}
\crefname{section}{Section}{Sections}
\crefname{figure}{Figure}{Figures}
\crefname{quest}{Question}{Questions}

\newcommand{\R}{\mathbb{R}}

\newtheorem{thm}{Theorem}[section]
\newtheorem{lemma}[thm]{Lemma}

\theoremstyle{definition}
\newtheorem{definition}[thm]{Definition}

\begin{document}
\begin{abstract}
Classical theorems from the early 20th century
state that
any Haar measurable homomorphism between locally compact groups is continuous. In particular, any Lebesgue-measurable homomorphism $\phi:\mathbb{R} \to \mathbb{R}$ is of the form $\phi(x)=ax$ for some $a \in \mathbb{R}$. In this short note, we prove that any Lebesgue measurable function $\phi:\R \to \R$ that vanishes under any $d+1$ ``difference operators''  is a polynomial of degree at most $d$. More generally, we prove the continuity of any Haar measurable polynomial map  between locally compact groups, in the sense of Leibman. We deduce the above result as a direct consequence of a theorem about the automatic continuity of cocycles. 
 \end{abstract}

\maketitle

\section{Statement of results}

The Cauchy functional equation 
\begin{equation}
    \phi(x+y)= \phi(x)+\phi(y) \mbox{ for all }x,y, 
\end{equation}
is a well-known and fundamental equation in the theory of functional equations. 
In 1913,  it was shown by   Fr{\'e}chet  that any Lebesgue measurable solution  $\phi:\R \to \R$ of the Cauchy functional equation is continuous, thus takes the form $\phi(x)=ax$ for some $a \in \R$ \cite{frechet1913pri}.  Other proofs have been published, for instance by Banach \cite{banach1920equation} and  Sierpi{\'n}ski \cite{sierpinski1920equation} in 1920. See for instance \cite{MR3627381} for a relatively recent survey of some results around the  Cauchy functional equation.
Given  $t \in \R$ and $f:\R \to \R$,
denote by $\Delta_t f:\R \to \R$ the function given by
\[
\Delta_t f(x)= f(x+t) - f(x).
\]

A function $\phi:\R \to \R$ satisfies Cauchy's functional equation if and only if $\phi(0)=0$ and 
\[
\Delta_x \Delta_y \phi = 0 \mbox{ for all }x,y.
\]
 Fr{\'e}chet's theorem can thus be viewed as the case $d=1$ of the following theorem:
\begin{thm}\label{thm:poly_R_diff}
    Let $f:\R \to \R$ be a Lebesgue measurable function. Suppose that there exists $d \ge 0$ that satisfies
    \begin{equation}\label{eq:D_n_f_const}
    \Delta_{t_d}\ldots \Delta_{t_0} f(x) =0 \mbox{ for all } x,t_0,\ldots,t_d \in \R.    
    \end{equation}
    
    Then $f$ is a polynomial of degree at most $d$. Namely, there exists $a_0,\ldots, a_d \in \R$ such that
    \[ f(x) =a_0 + a_1x +\ldots a_d x^d.\]
\end{thm}

By elementary considerations,  Fr{\'e}chet's theorem is equivalent to the statement that any  Lebesgue measurable homomorphism of $\R$ is continuous.  Improving on    Sierpi{\'n}ski's work,  Steinhaus showed that for any Lebesgue measurable set $A \subset \mathbb{R}$ of positive measure, $A-A$ is a neighborhood of $0$ \cite{steinhaus1937distances},  leading  to another short proof of Fr{\'e}chet's theorem.

 Weil extended  Steinhaus's result and  deduced that any Haar measurable homomorphism between locally compact Polish groups is continuous. 
The phenomena of continuity following from apriori weaker assumptions is called ``automatic continuity''.
See Rosendal's survey \cite{MR2535429} for more on the subject of automatic continuity of homomorphisms.

We will use multiplicative notation for the group operation whenever we do not explicitly assume that the group is commutative, and additive notation whenever we do explicitly assume that the group is commutative.
 \Cref{thm:poly_R_diff} and   Weil's result can be unified  using Leibman's notion of \emph{polynomial maps} between groups, which we now recall \cite{MR1910931}:
\begin{definition}\label{def:diff_op}
Let $G, H$ be groups. 
Given an  element $g \in G$,   the \emph{difference operator} $\Delta_g$ on the space of functions from $G$ to $H$ is given by
\[
\Delta_g f(x) = f(gx)f(x)^{-1},\mbox{ for } f: G \to H \mbox{ and } x \in G.
\]
\end{definition} 
In the case of commutative groups $G, H$, when the group operations are written in additive notation, the difference operation $\Delta_g$ with respect to an element $g \in G$ can be written as follows:
\[
\Delta_g f(x) = f(x+g)-f(x),\mbox{ for } f: G \to H \mbox{ and } x \in G.
\]
\begin{definition}\label{def:Liebman_poly}
A \emph{Leibman polynomial of degree at most $d$ } is a function $P: G \to H$ that vanishes after applying any $d$ difference operators.
That is, $P:G \to H$ is   a Leibman map of degree at most $d$ if
\[
\Delta_{g_d}\ldots \Delta_{g_0} P(x) =1_H \mbox{ for all } x,g_0,\ldots,g_d \in G,
\]
where $1_H$ denotes  the identity element of $H$.
\end{definition}

Leibman polynomials of degree $0$ are precisely constant functions. A Leibman polynomial of degree $1$ is an ``affine homomorphism'', namely function $P: G \to H$ of the form $P(x) = h \phi(x)$, where $\phi: G \to H$ is a homomorphism and $h \in H$ is a fixed element.

\begin{thm}\label{thm:poly_aut_cont}
Let $P: G \to H$ be a Haar measurable Leibman polynomial between locally compact  Polish groups $G, H$. Then $P$ is continuous.
\end{thm}

In some cases, there is an explicit description of the (continuous) Leibman polynomials. For instance, it is known that any Leibman polynomial from a group $G$ into a nilpotent group $H$ factors through a nilpotent quotient of $G$ \cite{MR1910931}. See also \cite{MR2877065} for a refinement of this statement and \cite[Proposition 4.2]{MR3581232} for a quick proof in the case that $H$ is abelian. 
It is also known that  Leibman polynomials between finitely generated nilpotent groups are in fact ordinary polynomials in the Mal'cev coordinates \cite{MR1910931}. From this, it follows for instance that a real-valued continuous Leibman polynomial from a connected nilpotent group of real matrices is nothing but an ordinary polynomial in the entries of the matrix. 

In the case where $G=V$ and $H=W$ are locally compact topological vector spaces over a field of characteristic $0$, Theorem \ref{thm:poly_aut_cont} is a simple corollary of  Fr{\'e}chet's Theorem.
\begin{proof}[Proof of Theorem \ref{thm:poly_aut_cont} when $G=V$ and $H=W$ (sketch) :]
    The proof uses induction on the degree of the polynomial. The degree $0$ case is trivial and the degree $1$ case is known by Fr{\'e}chet's Theorem for the real case and Weil's Theorem for the general locally compact case. 
    Let $P: V\to W$ be a measurable Leibman polynomial of degree $d$. We can represent \[P(x) = \Phi(x,x,x,\dots,x) + P_{d-1}(x)\] 
    where $P_{d-1}:V\to W$ is a Leibman polynomial of degree $d-1$ and $\Phi:V^d\to W$ is the $d$-additive Haar measurable map given by $\Phi(x_1,\dots,x_d) = \frac{1}{d!}\Delta_{x_1}\ldots \Delta_{x_d} P$. 
    Note that the right-hand side is a constant since $P$ is of degree $d$. By  Weil's Theorem, we deduce that $\Phi$ depends continuously on each coordinate, and from here (using  $d$-additivity of $\Phi$) it is not complicated to deduce its continuity. By induction hypothesis, we deduce that $P_{d-1}$ is continuous, and hence $P$ is continuous as well.
\end{proof}
Using the fact that composition of polynomial maps between nilpotent groups is again polynomial (see \cite[Prop. 3.22]{MR1910931}) one can use the Mal'cev coordinates parametrization of Leibman polynomials to deduce the special case of Theorem \ref{thm:poly_aut_cont}  for Leibman polynomials between nilpotent Lie groups, again by applying Fr{\'e}chet's Theorem directly.

Obtaining the general case of 
\Cref{thm:poly_aut_cont} can seem slightly trickier. However, \Cref{thm:poly_aut_cont} is  an immediate corollary of the following slightly more general result:
\begin{thm}\label{thm:diff_aut_cont}
    Let $f: G \to H$ be a Haar measurable function between locally compact Polish groups $G, H$.
   If $\Delta_g f: G \to H$ is continuous for every $g \in G$, 
    then $f: G \to H$ also continuous.
\end{thm}
Again, Weil's theorem is a particular case because for any homomorphism $\phi: G \to H$ $\Delta_g \phi: G \to H$ is a constant function for every $g \in G$, hence continuous.
\Cref{thm:poly_aut_cont} follows directly from \Cref{thm:diff_aut_cont} by induction on $d$. 

We will introduce one more level of generalization: 
Let $G$,$H$ be groups, and let $G \curvearrowright X$ be a continuous action of $G$ on a set $X$. By a continuous action we mean that there is a homomorphism from the group $G$ into the group of homeomorphisms of $X$, so that the map $G \times X \to X$ given by $(g,x) \mapsto g\cdot x$ is continuous. 
A \emph{cocycle} with respect to this action is a function $c: G \times X \to H$ 
that satisfies
\begin{equation}\label{eq:coycle_eq}
    c(g_1g_2,x) = c(g_1, g_2\cdot x) c(g_2,x).
\end{equation} 
Any function $f:X \to H$  defines a cocycle $\Delta f: G \times X \to H$  given by
\begin{equation}
\Delta f (g,x) = f(g\cdot x) f^{-1}(x) \mbox { for } g \in G, x\in X.   
\end{equation}

The map  $f \mapsto \Delta f$ is called the \emph{coboundary map}. It is easy to check that $c=\Delta f$ satisfies \eqref{eq:coycle_eq}. A cocycle of the form $\Delta f$ is called a \emph{coboundary}.
For a function $f: G \to H$, we denote by $\Delta f: G \times G \to H$ the coboundary associated with the action of $G$ on itself by multiplication from the left.

From \Cref{eq:coycle_eq} it follows that for any cocycle $c: G \times X \to H$ $c(1_G,x) = 1_H$ for all $x \in X$ and also that for any $(g,x) \in G \times X$ the following holds:
\[ c(g^{-1},x)= \left( c(g,g^{-1}\cdot x)  \right)^{-1}.\]

\begin{thm}\label{thm:aut_cont_cocycle}
    Let $G, H$ be locally compact Polish groups, and let $X$ be a locally compact Polish space.
    Suppose that $G$ acts on $X$ continuously,
    and let $c: G \times X \to H$ be a  cocycle 
    such that $x \mapsto c(g,x)$ is continuous 
    for every $g \in G$  and so that $g \mapsto c(g,x)$ is Haar measurable for every $x \in X$.
    Then $c: G \times X \to H$ is continuous.
\end{thm}

In that case where $X=\{x_0\}$ is a singleton, any cocycle $c: G \times X \to H$ is of the form $c(g,x_0)=\phi(g)$, where $\phi$ is a homomorphism. Thus, the case where $X$ is a singleton, \Cref{thm:aut_cont_cocycle} coincides with Weil's theorem on automatic continuity of homomorphisms.
\Cref{thm:diff_aut_cont} follows directly by applying \Cref{thm:aut_cont_cocycle} to the coboundary $\Delta f: G \times G \to H$.

Somewhat different aspects of continuity for cocycles on topological groups have been explored in the literature. In \cite{MR2995370} Becker has shown that whenever $G, H$ are Polish groups and $X$ is a Polish space, any Borel measurable cocycle $c: G \times X \to H$ is \emph{nearly essentially continuous} in the sense that for any Borel probability measure on $X$ which is $G$-quasi-invariant and any $\epsilon >0$ there exists a compact set $K \subseteq X$ such that $\mu(K) > 1- \epsilon$ and a Borel subset $S \subseteq G \times X$ such that $(g,x) \in S$ whenever $x, gx \in K$ and so that the restriction of $c$ to $K$ is continuous \cite[Theorem 1.4.5]{MR2995370}. Several other results regarding the continuity of cocycles are discussed in \cite{MR2995370} and earlier references within. 
As far as we are aware, \Cref{thm:aut_cont_cocycle} is not directly covered by existing results in the literature.

We remark that in this paper we do not explore possible extensions beyond the setup of locally compact Polish groups. 
Automatic continuity holds in the context of Haar measurable homomorphisms between (not necessarily Polish) locally compact groups \cite{MR2983459,MR1069290}.
Answering a longstanding problem originating in Christensen’s seminal work on Haar null sets, Rosendal proved that any universally measurable homomorphism between Polish groups is automatically continuous \cite{MR3996719}. Rosendal's proof utilizes a ``Steinhaus type principle'' due to  Christensen \cite{MR0308322}. 
One can wonder if   \Cref{thm:aut_cont_cocycle} can be similarly extended beyond the setting of locally compact Polish groups.

\emph{Acknowledgements:} The authors thank Nachi Avraham Re'em, Uri Bader, Eli Glasner, Yair Glasner, and the anonymous referee
for useful references, historical remarks 
and helpful comments.
 The first author is supported by the Israel Science Foundation grant no. 985/23.
The second author is supported by ERC 2020 grant HomDyn (grant no.~833423).

\section{Continuity of cocycles}

We proceed to prove \Cref{thm:aut_cont_cocycle}.
The proof is based on the following theorem of Andr\'e Weil, which is  a generalization of  Steinhaus's theorem \cite{steinhaus1937distances} (see for instance \cite[Theorem 2.3]{MR2535429}):
\begin{thm}[H. Steinhaus, A. Weil]\label{thm:Weil_product_open}
    Let $G$ be a locally compact Polish group and $E \subseteq G$ a Haar measurable set of positive Haar measure. 
     Then the identity element of $G$ is contained in the interior of $EE^{-1}$.
\end{thm}

We first introduce  a simple measure-theoretic lemma that does not involve groups:

\begin{lemma}\label{lem:pos_measure_small_diameter_image}
    Let $Y, Z$ be separable metric spaces, with $Z$ locally compact. 
    Let $\mu$ be a Radon measure on $Z$ and let $F: Z\times Z \to Y$ be a  function  such that $x_2 \mapsto F(x_1,x_2)$ is continuous for every $x_1 \in Z$ and $x_1 \mapsto F(x_1,x_2)$ is $\mu$-measurable for every $x_2 \in X$. Then for any $\epsilon >0$ and any measurable set $E_0 \subseteq Z$ with $\mu(E_0) >0$, there exists a measurable set $E \subset E_0$ with $\mu(E)>0$ such that the diameter of the image of $E \times E$ under $F$ is at most $\epsilon$ in the following sense:
    \[
    d_Y(F(x_1,x_2), F(\tilde x_1,\tilde x_2)) < \epsilon \mbox{ for every } x_1,x_2,\tilde x_1,\tilde x_2 \in E.
    \]
\end{lemma}
\begin{proof}
    Let $Z$, $\mu$ and $F: Z \times Z \to Y$ be as in the statement. Fix $\epsilon >0$ and a  measurable set $E_0 \subseteq Z$ with $\mu(E_0) >0$.
    Because $\mu$ is a Radon measure, we can find a compact set $K\subseteq E_0$ such that 
   $\mu(K)>0$.
   Because $K$ is a compact metric space, it follows that the space $C(K, Y)$ of continuous functions from $K$ to $Y$ is separable.
   Given $f \in C(K,Y)$ and $\delta >0$, let
   \[
   A_{f,\delta} = \left\{\tilde f \in C(K,Y)~:~ d_Y(\tilde f(x),f(x)) \le \delta ~\forall~ x\in K \right\},
   \]
   and let 
   \[
   E_{f,\delta} =\left\{ x \in K~:~ 
   F_x \in A_{f,\delta}
   \right\}.
   \]
   where for $x \in Z$, the map  $F_x \in C(K,Y)$ is given by $F_x(\tilde x) = F(x,\tilde x)$. 
   Choose $\delta = \epsilon/4$. 
   
   Let us explain why the set $E_{f,\delta} \subseteq K$ is $\mu$-measurable.
   Since $Z$ is separable, so is $K$. Let  $\{x_1,\ldots,x_n,\ldots\}  \subseteq K$ be a dense countable set.
   By continuity of $f$ and $F_x$,
   the set $E_{f,\delta} \subseteq K$ can be written as
   \[
   E_{f,\delta} = \bigcap_{i=1}^\infty\left\{
   x \in K~:~ d_Y(F(x,x_i),f(x_i))\le \delta
   \right\}.
   \]
   Because $x \mapsto F(x,x_i)$ is $\mu$-measurable for every $i \in \mathbb{N}$, we see that $E_{f,\delta}$ is a countable intersection of $\mu$-measurable sets.

   Then for every $f \in C(K, Y)$ the set $A_{f,\delta}\subseteq C(K, Y)$ is a closed neighborhood of $f$. Since $C(K,Y)$ is separable there is a countable collection  $T \subseteq C(K,Y)$ such that $C(K,Y) = \bigcup_{f\in T} A_{f,\delta}$. Hence $K = \bigcup_{f \in T}E_{f,\delta}$. 
   Because $\mu(K)>0$, there exists $f \in T$ such that $\mu(E_{f,\delta})>0$. Fix such $f$. Since  $f$ is continuous and $K$ is compact, there exists a finite cover open cover $\mathcal{U}$ of $K$ such that for every $U \in \mathcal{U}$ and every $x_1,x_2 \in U$ we have $d_Y(f(x_1),f(x_2)) < \delta$. Because $\mu(E_{f,\delta})>0$ and
   $\mathcal{U}$ covers $K$, there exists $U \in \mathcal{U}$ such that $\mu(E_{f,\delta} \cap U) > 0$.
   Choose $E= E_{f,\delta} \cap U$. Clearly, $E \subseteq E_0$ is measurable and $\mu(E)>0$.
   Now for any $x_1,x_2,\tilde x_1,\tilde x_2 \in E$, we have 
   \[
   d_Y(F(x_1,x_2),F(\tilde x_1,\tilde x_2)) \le
   d_Y(F(x_1,x_2),f(x_2))+ d_Y(f(x_2),f(\tilde x_2)) + d_Y(f(\tilde x_2),F(\tilde x_1,\tilde x_2)).
   \]
   Since $x_1,x_2 \in E \subseteq E_{f,\delta}$ and $x_2,\tilde x_2 \in K$
   we have that $d_Y(F(x_1,x_2),f(x_2))< \delta$ and $d_Y(f(\tilde x_2),F(\tilde x_1,\tilde x_2))< \delta$.
   Since $x_2,\tilde x_2 \in E \subseteq U$, we have that $d_Y(f(x_2),f(\tilde x_2))<\delta$.
   So altogether, $d_Y(F(x_1,x_2),F(\tilde x_1,\tilde x_2))< 3\delta < \epsilon$.
\end{proof}

\begin{proof}[Proof of \Cref{thm:aut_cont_cocycle}]
 Let $c: G \times X \to H$ be a 
 cocycle that satisfies the conditions in the statement of \Cref{thm:aut_cont_cocycle}
. 
We first prove that for every $x_0 \in X$ the function $c: G \times X \to H$ is continuous at $(1_G,x_0)$.   
We need to prove the following: For every $x_0 \in X$ and every open neighborhood $V \subseteq H$ of $1_H$ there exist open sets $U \subseteq G$ and $W \subseteq X$ with $(1_G,x_0) \in U \times W$ so that  $U \times W \subseteq c^{-1}(V)$.
Fix $x_0 \in X$ and any open set $V \subseteq H$ with $1_H \in V$. 

We fix  metrics $d_H,d_H$ and $d_X$ on $G$, $X$, and $H$ respectively that are compatible with their Polish topology, and so that any closed ball is compact (see for instance \cite{MR0891165}).
For every point $g\in G$ and $r>0$ denote by $B_g(r)$ the open ball of radius $r$ around $g$ with respect to $d_G$. Similarly, for $x \in X$ and $h \in H$, $B_x(r)$ and $B_h(r)$ denote the radius $r$ ball with respect to $d_X$ and $d_H$ around $x$ and $h$ respectively.
Find $\epsilon >0$ such that $V\supseteq B_{1_H}(\varepsilon)$.

Since $X$ is locally compact, we can find a compact subset $X_0 \subseteq X$ that contains $x_0$ in its interior.
For every $g \in G$ there exists $r_g > 0$ so that $c(g, x)\in \overline{B_{1_H}(r_g)}$ for every $x \in X_0$.
 Since $r_g$ can be chosen by 
\[r_g = \sup_{i \in \mathbb{N}} d_H(c(g,x_i), 1_H),\]
where $(x_i)_{i=1}^\infty$ is a dense subset of $X_0$, we deduce that $g\mapsto r_g$ is measurable. 
Since $G = \bigcup_{r>0}\{g\in G:r_g < r\},$ we deduce that  there exists $r_0>0$ so that $E_0 = \{g\in G:r_g < r_0\}$ has positive Haar measure. 

Since the closed ball $\overline{B_{1_H}(r_0)}$ is compact and multiplication in $H$ is continuous, there exists $\delta > 0$ such that for any  $h \in \overline{B_{1_H}(r_0)}$, we have $B_h(\delta) (B_h(\delta))^{-1}\subseteq B_{1_H}(\varepsilon)$. 
Set $Y = H$ and $Z=G$. Define $F:Z \times Z \to Y$ by $F(g_1,g_2)= c(g_1,g_2(x_0))$. 
Then we can apply \Cref{lem:pos_measure_small_diameter_image} and find a set $E_1 \subseteq E_0$ of positive Haar measure and a closed ball $\tilde B = \overline{B_{h_0}(\delta/2)}$  with $h_0\in B_{1_H}(r_0)$ such that $c(g_1,g_2(x_0)) \in \tilde B$ for all $g_1,g_2 \in E_1$.

Let $E_2$ be a compact subset of $E_1$  of positive Haar measure. 
Fix $g_1\in E_2$. 
Since $c$ is continuous in the second argument, the map $(g_2, x)\mapsto c(g_1, g_2(x))$ is continuous, and hence 
\[U_{g_1} = \{(g_2, x)\in G\times X_0: c(g_1, g_2(x)) \in B_{h_0}(\delta)\}\] is open. 
Since $E_2 \times \{x_0\}\subseteq U_{g_1}$ there is $\tilde \zeta_{g_1}>0$ such that 
$E_2\times B_{x_0}(\tilde \zeta_{g_1}) \subseteq U_{g_1}$.
Let 
\[\zeta_{g_1} = \inf\{d(x_i, x_0): i=1,2,\dots, j = 1,\dots, \text{ so that }c(g_1, g_{2, j}(x_i)) \notin B_{h_0}(\delta)\}/2,\]
where $(x_i)_{i=1}^\infty$ is a dense subset of $X_0$ and $(g_{2,j})_{j=1}^\infty$ is a dense subset of $E_2$. 
Then $\zeta_{g_1}$ depends measurably on $g_1$ and satisfies that 
\begin{enumerate}
  \item $\zeta_{g_1} \ge \tilde \zeta_{g_1}/2 > 0$;
  \item for every $g_2 \in E_2, x \in B_{x_0}(\zeta_{g_1})$ we have $c(g_1, g_{2}(x)) \in \overline{B_{h_0}(\delta)}$.
\end{enumerate}
Let $B = \overline{B_{h_0}(\delta)}$. 
Let $\zeta>0$ be a number such that $E = \{g_1\in E_2: \zeta_{g_1} > \zeta\}$ has positive Haar measure. 
Denoting $W=B_{x_0}(\zeta)$, we obtain that for every $x \in W$ and every $g_1,g_2 \in E$ we have $c(g_1,g_2(x)) \in B$.

Let $U = E E^{-1} \subseteq G$.
By Weil's  \Cref{thm:Weil_product_open}, $U$ is a neighborhood of $1_G$.
Now suppose $g \in U$ and $x \in W$. Then there exists $g_1,g_2 \in E$ such that $g= g_1 g_2^{-1}$, and so
\[
c(g,x)=c(g_1g_2^{-1},x)=c(g_1,g_2^{-1}\cdot x)\left( c(g_2,g_2^{-1}\cdot x)\right)^{-1} \in B B^{-1} \subseteq V.\]

This completes the proof that $c: G \times X \to H$ is continuous at $(1_G,x_0)$ for every $x_0 \in X$.
To complete the proof that $c: G \times X \to H$ is continuous at every point $(g_0,x_0) \in G \times X$ recall the cocycle equation
\[c(gg_0,x)= c(g,g_0\cdot  x)c(g_0,x).\]
So 
\[
\lim_{(g,x) \to (g_0,x_0)}c(g,x) = 
\lim_{(g,x) \to (1_G,x_0)}c(gg_0,x)=\]
\[
\lim_{(g,x) \to (1_G,x_0)} c(g,g_0 \cdot x)c(g_0,x) =
\lim_{(g,x) \to (1_G,x_0)} c(g,g_0 \cdot x)\cdot  \lim_{(g,x) \to (1_G,x_0)} c(g_0,x) =
1_H \cdot c(g_0,x_0).
\]
\end{proof}

\bibliographystyle{amsplain}
\bibliography{library}

\providecommand{\bysame}{\leavevmode\hbox to3em{\hrulefill}\thinspace}
\providecommand{\MR}{\relax\ifhmode\unskip\space\fi MR }
\providecommand{\MRhref}[2]{%
  \href{http://www.ams.org/mathscinet-getitem?mr=#1}{#2}
}
\providecommand{\href}[2]{#2}
\begin{thebibliography}{10}

\bibitem{banach1920equation}
Stefan Banach, \emph{Sur l'{\'e}quation fonctionnelle $f (x+ y)= f (x)+ f
  (y)$}, Fundamenta Mathematicae \textbf{1} (1920), no.~1, 123--124.

\bibitem{MR2995370}
Howard Becker, \emph{Cocycles and continuity}, Trans. Amer. Math. Soc.
  \textbf{365} (2013), no.~2, 671--719. \MR{2995370}

\bibitem{MR0308322}
Jens Peter~Reus Christensen, \emph{Borel structures in groups and semigroups},
  Math. Scand. \textbf{28} (1971), 124--128. \MR{308322}

\bibitem{frechet1913pri}
Maurice Fr{\'e}chet, \emph{Pri la funkcia equacio $f (x+ y)= f (x)+ f (y)$},
  Enseignement Math. \textbf{15} (1913), 390--393.

\bibitem{MR2877065}
Ben Green and Terence Tao, \emph{The quantitative behaviour of polynomial
  orbits on nilmanifolds}, Ann. of Math. (2) \textbf{175} (2012), no.~2,
  465--540. \MR{2877065}

\bibitem{MR1069290}
Adam Kleppner, \emph{Correction to: ``{M}easurable homomorphisms of locally
  compact groups'' [{P}roc. {A}mer. {M}ath. {S}oc. {\bf 106} (1989), no. 2,
  391--395; {MR}0948154 (89k:22005)]}, Proc. Amer. Math. Soc. \textbf{111}
  (1991), no.~4, 1199--1200. \MR{1069290}

\bibitem{MR2983459}
Yulia Kuznetsova, \emph{On continuity of measurable group representations and
  homomorphisms}, Studia Math. \textbf{210} (2012), no.~3, 197--208.
  \MR{2983459}

\bibitem{MR1910931}
A.~Leibman, \emph{Polynomial mappings of groups}, Israel J. Math. \textbf{129}
  (2002), 29--60. \MR{1910931}

\bibitem{MR3581232}
Tom Meyerovitch, Idan Perl, Matthew Tointon, and Ariel Yadin, \emph{Polynomials
  and harmonic functions on discrete groups}, Trans. Amer. Math. Soc.
  \textbf{369} (2017), no.~3, 2205--2229. \MR{3581232}

\bibitem{MR3627381}
Daniel Reem, \emph{Remarks on the {C}auchy functional equation and variations
  of it}, Aequationes Math. \textbf{91} (2017), no.~2, 237--264. \MR{3627381}

\bibitem{MR2535429}
Christian Rosendal, \emph{Automatic continuity of group homomorphisms}, Bull.
  Symbolic Logic \textbf{15} (2009), no.~2, 184--214. \MR{2535429}

\bibitem{MR3996719}
\bysame, \emph{Continuity of universally measurable homomorphisms}, Forum Math.
  Pi \textbf{7} (2019), e5, 20. \MR{3996719}

\bibitem{sierpinski1920equation}
Wac{\l}aw Sierpi{\'n}ski, \emph{Sur l'{\'e}quation fonctionnelle f (x+ y)= f
  (x)+ f (y)}, Fundamenta Mathematicae \textbf{1} (1920), no.~1, 116--122.

\bibitem{steinhaus1937distances}
Hugo Steinhaus, \emph{Sur les distances des points dans les ensembles de mesure
  positive}, Seminarium Matematyczne, 1937.

\bibitem{MR0891165}
Robert Williamson and Ludvik Janos, \emph{Constructing metrics with the
  {H}eine-{B}orel property}, Proc. Amer. Math. Soc. \textbf{100} (1987), no.~3,
  567--573. \MR{891165}

\end{thebibliography}
\end{document}